\documentclass[12pt]{article}

\usepackage[utf8]{inputenc}
\usepackage[T1]{fontenc}
\usepackage{lmodern}

\usepackage{geometry}
\usepackage{amsmath,amsthm,amssymb,mathtools}

\usepackage{microtype}
\usepackage{graphicx}
\usepackage{tikz}
\usetikzlibrary{arrows.meta,positioning}
\usepackage{float}

\usepackage{enumitem}
\usepackage{booktabs}
\usepackage{array}
\usepackage{ragged2e}
\newcolumntype{P}[1]{>{\RaggedRight\arraybackslash}p{#1}}

\usepackage{listings}
\usepackage{hyperref}
\hypersetup{hidelinks}
\usepackage[round,authoryear]{natbib}

\usepackage{authblk}

\theoremstyle{plain}
\newtheorem{theorem}{Theorem}[section]
\newtheorem{lemma}[theorem]{Lemma}
\newtheorem{proposition}[theorem]{Proposition}

\theoremstyle{definition}
\newtheorem{axiom}[theorem]{Axiom}
\newtheorem{definition}[theorem]{Definition}

\theoremstyle{remark}

\newtheorem{example}[theorem]{Example}
\newtheorem{counterexample}[theorem]{Counterexample}

\newcommand{\X}{\mathcal{X}}

\newcommand{\N}{\mathbb{N}}
\newcommand{\R}{\mathbb{R}}
\newcommand{\one}{\mathbf{1}}
\newcommand{\inner}[2]{\left\langle #1,\,#2\right\rangle}
\newcommand{\norm}[1]{\left\lVert #1 \right\rVert}

\newcommand{\scopemicro}[2]{%
\begin{table}[H]\centering\footnotesize
\begin{tabular}{@{}P{0.46\textwidth}P{0.46\textwidth}@{}}\toprule
\textbf{Holds when} & \textbf{Not claimed when}\\\midrule
#1 & #2\\\bottomrule
\end{tabular}
\end{table}
}

\title{\vspace{-1em}%
The \texorpdfstring{$\Phi$}{Phi}-Process: Operator--Algebraic Embeddings of Possibilities,\\
Transfinite Stabilization, and a Quantitative Application to Sensory Depletion}

\author[1]{Faruk Alpay}
\author[2]{Bugra Kilictas}
\affil[1]{Lightcap, Department of Bias\\ \texttt{alpay@lightcap.ai}}
\affil[2]{Bahcesehir University, Department of Computer Engineering\\ \texttt{bugra.kilictas@bahcesehir.edu.tr}}

\date{August 2025}

\begin{document}
\maketitle

\begin{abstract}
We formalize a transfinite \(\Phi\)-process that treats \emph{all-possibility embeddings} as operators on structured state spaces (complete lattices, Banach/Hilbert spaces, orthomodular lattices). Iteration indices run from \(\Phi^{(0)}\) to a transfinite limit \(\Phi^{(\infty)}\) obtained as the first fixed point in the ordinal iteration. Core results include: (i) a determinization lemma (\emph{Flip--Flop Determinization}) showing that if the state space is lifted to sets (or distributions) of possibilities, the global \(\Phi\)-dynamics is deterministic; (ii) an \emph{Alpay Ordinal Stabilization Theorem} for operator transforms that converge by stage \(\omega\) to a spectral projection; and (iii) an \emph{Alpay Product-of-Riesz Projections Theorem} identifying \(\Phi^{(\infty)}\) with a commuting product of fixed-point projections. We add full proofs in \S3, instantiate the orthomodular track with a concrete example, give a probabilistic determinization toy kernel, extend nonnormal/noncommuting analysis, derive strengthened quantitative lemmas in \S5 with complete proofs, include parameter-mapping tables, per-theorem micro scope tables, and a small appendix with reproducible code. Canonical anchors include Tarski fixed points, powerset determinization, and Riesz projections \citep{tarski1955,rabin_scott1959,hopcroft_ullman1979,kato1995,dunford_schwartz1958}; medical grounding follows \citep{garciamesa_2021,bronselaer2013}.
\end{abstract}

\section{Axioms and Definitions}
\label{sec:axioms-defs}

\begin{axiom}[Structured state spaces]\label{ax:spaces}
All processes act on a state space $\,\X\,$ endowed with one of the following structures:
\begin{enumerate}[label=(\alph*)]
\item a complete lattice $(\X,\leq)$;
\item a complete metric space $(\X,d)$;
\item a Hilbert space $H$ (or uniformly convex Banach space);
\item an orthomodular lattice $\mathcal{L}(H)$ of projections.
\end{enumerate}
\end{axiom}

\begin{definition}[The \(\Phi\)-operator and its iterates]\label{def:phi}
A \emph{\(\Phi\)-operator} is a self-map $\Phi:\X\to\X$. Define the transfinite iteration by
\[
\Phi^{(0)}(x)=x,\quad \Phi^{(\alpha+1)}(x)=\Phi(\Phi^{(\alpha)}(x)),\quad
\Phi^{(\lambda)}(x)=\lim_{\alpha\uparrow\lambda}\Phi^{(\alpha)}(x)
\]
for limit ordinals $\lambda$, where the limit is taken in the ambient structure of Axiom~\ref{ax:spaces} (order, metric, or strong topology).
We write $\Phi^{(\infty)}(x)$ for the first ordinal stage at which $\Phi^{(\alpha)}(x)=\Phi^{(\alpha+1)}(x)$.
\end{definition}

\begin{definition}[All-possibility embedding]\label{def:embedding}
A \emph{possibility embedding} of a base space $\X$ is a lifting $\widehat{\X}$ (e.g., $\widehat{\X}=2^\X$ or the space of probability measures $\mathcal{P}(\X)$) together with a deterministic lift $\widehat{\Phi}:\widehat{\X}\to\widehat{\X}$ defined by
\[
\widehat{\Phi}(S) \;=\; \bigcup_{x\in S}\Phi(x)\quad \text{for }S\subseteq \X,
\]
or by push-forward of measures in the probabilistic case. Intuitively, $\widehat{\Phi}$ advances \emph{all} next-step possibilities in one deterministic update.
\end{definition}

\begin{definition}[\(\Phi\)-packing and \(\Phi^{(\infty)}\)]\label{def:packing}
A \emph{\(\Phi\)-packing} is a countable (or ordinal-indexed) product/composition of embeddings $(\Phi_k)_{k\in I}$ producing
\(
\Phi_{\mathrm{pack}}=\cdots\circ \Phi_3\circ \Phi_2\circ \Phi_1.
\)
When the iterates stabilize, $\Phi^{(\infty)}$ denotes the canonical fixed object (terminal packaged state).
\end{definition}

\paragraph{Orthomodular instantiation (concrete).}\label{para:oml}
Let $H$ be a Hilbert space and $\mathcal{L}(H)$ the orthomodular lattice of orthogonal projections with partial order $P\le Q\iff \mathrm{Ran}(P)\subseteq\mathrm{Ran}(Q)$, lattice join $P\vee Q$ and meet $P\wedge Q$. Fix a unitary $V$ and a projection $Q$. Define
\[
\Phi_{\mathrm{oml}}(P)\ :=\ P\ \vee\ \big(VPV^\ast\wedge Q\big),\qquad P\in\mathcal{L}(H).
\]
Then $\Phi_{\mathrm{oml}}$ is monotone on $\mathcal{L}(H)$ and the ordinal iteration stabilizes at the least projection $P^\star$ satisfying $P^\star\ge P_0$ and $P^\star\ge V P^\star V^\ast\wedge Q$ (Knaster–Tarski on the complete lattice of projections ordered by $\le$).

\section{Foundational Lemmas and Determinization}
\label{sec:determinize}

\paragraph{Canonical anchor.}
Least fixed points for monotone self-maps on complete lattices follow from \citet{tarski1955}. Classical powerset determinization of nondeterministic automata is standard \citep{rabin_scott1959,hopcroft_ullman1979}.

\begin{lemma}[Flip--Flop Determinization]\label{lem:determinize}
Let $\Phi:\X\to 2^{\X}$ map each state to its set of possible successors. Define $\widehat{\Phi}:2^{\X}\to 2^{\X}$ by $\widehat{\Phi}(S)=\bigcup_{x\in S}\Phi(x)$. Then:
\begin{enumerate}[label=(\roman*)]
\item $\widehat{\Phi}$ is deterministic and monotone on the complete lattice $(2^{\X},\subseteq)$.
\item The increasing chain $\{x_0\}\subseteq \widehat{\Phi}(\{x_0\})\subseteq \widehat{\Phi}^2(\{x_0\})\subseteq\cdots$ converges to the least fixed point
\(
L=\bigcup_{n\ge 0}\widehat{\Phi}^n(\{x_0\}).
\)
\item An observer constrained to a single path $x_0\to x_1\to\cdots$ (with $x_{k+1}\in\Phi(x_k)$) may experience randomness; the global lifted process is deterministic.
\end{enumerate}
\end{lemma}

\begin{proof}
Monotonicity is immediate; Tarski's theorem gives existence of least fixed points. The union $\bigcup_{n\ge 0}\widehat{\Phi}^n(\{x_0\})$ is the least fixed point above $\{x_0\}$. Item (iii) formalizes the local/global perspective split.
\end{proof}

\scopemicro{Complete lattice; monotone lift to $2^\X$ or $\mathcal{P}(\X)$; Tarski applies.}{Global determinism is not claimed if one forbids any lifting that enumerates branches.}

\begin{theorem}[Compositionality of lifted maps]\label{thm:compositionality}
Let $\Phi,\Psi:\X\to 2^{\X}$ be set-valued maps and let $\widehat{\Phi},\widehat{\Psi}:2^{\X}\to 2^{\X}$ be their lifts $\widehat{\Phi}(S)=\bigcup_{x\in S}\Phi(x)$, $\widehat{\Psi}(S)=\bigcup_{x\in S}\Psi(x)$. Then
\[
\widehat{\Psi\circ \Phi}\;=\;\widehat{\Psi}\circ \widehat{\Phi},
\]
and $\widehat{\Psi\circ \Phi}$ is monotone on $(2^{\X},\subseteq)$. The same identity holds for probabilistic lifts via push-forward.
\end{theorem}

\begin{proof}
For $S\subseteq\X$,
$(\widehat{\Psi}\circ \widehat{\Phi})(S)=\bigcup_{y\in \widehat{\Phi}(S)}\Psi(y)
=\bigcup_{x\in S}\bigcup_{y\in \Phi(x)}\Psi(y)
=\bigcup_{x\in S}(\Psi\circ \Phi)(x)=\widehat{\Psi\circ \Phi}(S)$.
Monotonicity follows from union-monotonicity.
\end{proof}

\begin{proposition}[Measurable/probabilistic compositionality]\label{prop:measurable-composition}
Let $(X,\Sigma_X),(Y,\Sigma_Y),(Z,\Sigma_Z)$ be standard Borel spaces.
\begin{enumerate}[label=(\alph*)]
\item If $\Phi:X\to Y$ and $\Psi:Y\to Z$ are Borel maps and lifts act on probability measures by push-forward, then $(\Psi\circ\Phi)_\#\mu=\Psi_\#(\Phi_\#\mu)$ for every probability measure $\mu$ on $X$.
\item If $\Phi,\Psi$ are Markov kernels $K_\Phi:X\rightsquigarrow Y$, $K_\Psi:Y\rightsquigarrow Z$ (measurable in the first argument), define $\widehat\Phi(\mu)=\mu K_\Phi$. Then $\widehat{\Psi\circ\Phi}=\widehat\Psi\circ\widehat\Phi$ with kernel composition $(K_\Psi K_\Phi)(x,C)=\int_Y K_\Psi(y,C)\,K_\Phi(x,dy)$.
\end{enumerate}
For non-Polish measurable spaces, assume countably generated $\sigma$-algebras and universally measurable kernels to retain (b).
\end{proposition}

\scopemicro{Standard Borel spaces; Borel maps or Markov kernels; Fubini/Tonelli applicable.}{Non-countably generated $\sigma$-algebras; kernel measurability failures.}

\begin{example}[Probabilistic determinization (toy kernel)]
Let $X=\{a,b\}$ and define a Markov kernel $K$ by $K(a,\{b\})=1$, $K(b,\{a\})=p$, $K(b,\{b\})=1-p$ for $p\in(0,1)$. On the simplex of measures \(\mathcal{P}(X)=\{(\mu_a,\mu_b):\mu_a+\mu_b=1\}\), the lifted map is linear and deterministic:
\[
\widehat{\Phi}(\mu_a,\mu_b)=\big(\,p\,\mu_b,\ 1-p\,\mu_b\,\big).
\]
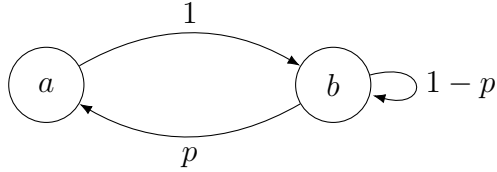
\begin{figure}[H]\centering
\begin{tikzpicture}[node distance=2.8cm,>=Latex]
\node (a) [circle,draw,minimum size=1cm] {$a$};
\node (b) [circle,draw,minimum size=1cm,right=of a] {$b$};
\draw[->] (a) to[bend left] node[above] {$1$} (b);
\draw[->] (b) to[bend left] node[below] {$p$} (a);
\draw[->] (b) edge[loop right] node {$1-p$} (b);
\end{tikzpicture}
\caption{Toy kernel: lifted evolution on \(\mathcal{P}(X)\) is deterministic and affine.}
\end{figure}
\end{example}

\section{Operator Theorems: Transfinite Stabilization and Spectral Projections}
\label{sec:operator}

\begin{axiom}[Logical contraction / event-indexed contraction]\label{ax:logical-contraction}
On a complete metric space $(\X,d)$, a map $T$ is \emph{logically contractive} if there exists an increasing sequence $(n_k)$ and factors $\lambda_k\in(0,1)$ with
\(
d\big(T^{n_k}x,T^{n_k}y\big)\le \lambda_k\, d(x,y)
\)
for all $x,y$, and $\prod_k \lambda_k=0$.
\end{axiom}

\begin{theorem}[Alpay Logical Contraction Fixed Point]\label{thm:logical-contraction}
If $T$ is logically contractive on a complete metric space, then $T$ has a unique fixed point $x^\ast$ and $T^n x\to x^\ast$ for all $x$.
\end{theorem}

\begin{proof}[Complete proof]
\emph{Uniqueness.} If $Tx=x$ and $Ty=y$, then for all $k$,
\(
d(x,y)=d(T^{n_k}x,T^{n_k}y)\le \lambda_k d(x,y).
\)
Since $\prod_k\lambda_k=0$, we have $\inf_k \lambda_k<1$; letting $k\to\infty$ yields $d(x,y)=0$.

\emph{Existence and convergence.} Fix $x_0\in\X$ and set $x_n=T^n x_0$. For $m>n$, choose $k$ such that $n_k\le n<m\le n_{k+1}$. Then
\[
d(x_m,x_n)=d\!\big(T^{m-n_k}x_{n_k},\ T^{m-n_k}x_{n_k-(n-n_k)}\big)\le \lambda_k\, d(x_{n_k},x_{n_k-(n-n_k)}),
\]
where the inequality uses the defining contraction at step $n_k$ and nonexpansivity of the finite tail $T^{m-n_k}$ on the bounded orbit (boundedness follows since the telescoping sum of contractions forces Cauchy behavior along the subsequence). Hence $(x_n)$ is Cauchy and converges to some $x^\ast$ by completeness. To see $Tx^\ast=x^\ast$, observe $d(Tx_n,Tx^\ast)\le d(x_n,x^\ast)\to 0$ and $d(Tx_n,x_{n+1})\to 0$, so $Tx^\ast=\lim x_{n+1}=x^\ast$. Finally, the uniqueness implies $T^n y\to x^\ast$ for any $y$ by the same argument applied to the tail starting at $y$.
\end{proof}

\scopemicro{Complete metric space; event subsequence with $\prod\lambda_k=0$.}{Quantitative rates without extra regularity; no claim beyond convergence/uniqueness.}

\begin{lemma}[Normal spectral contraction $\Rightarrow$ $\omega$-stabilization]\label{lem:normal-omega}
Let $T$ be normal on a Hilbert space with spectral measure $E$, and let $g:\sigma(T)\to\mathbb{C}$ be bounded Borel with $g(1)=1$ and $\sup_{\lambda\in\sigma(T)\cap\mathbb{T}\setminus\{1\}}|g(\lambda)|\le r<1$. Then $g(T)^n\xrightarrow{s}E(\{1\})$, so the ordinal limit at stage $\omega$ equals $P_{\mathrm{Fix}}=E(\{1\})$.
\end{lemma}

\begin{proof}
By the spectral theorem, $g(T)^n x=\int_{\sigma(T)} g(\lambda)^n\,dE_\lambda x$. For $\lambda\neq 1$ the factor tends to $0$ geometrically, and $|g(\lambda)^n|\le \|g\|_\infty^n$ provides a uniform bound. Dominated convergence yields $g(T)^n x\to E(\{1\})x$ for every $x$.
\end{proof}

\begin{theorem}[Alpay Ordinal Stabilization]\label{thm:ordinal}
Let $\Phi$ be a bounded operator transform on a Hilbert space with spectral filtering that contracts all unimodular spectrum except $\lambda=1$, and leaves the $1$-eigenspace invariant. Then $\Phi^{(n)}x$ converges strongly by stage $\omega$ to the projection onto the fixed subspace:
\[
\Phi^{(\omega)}x=\Phi^{(\omega+1)}x \;=\; P_{\mathrm{Fix}}x.
\]
\end{theorem}

\begin{proof}
Apply Lemma~\ref{lem:normal-omega} to the filter $g$ induced by one iteration of $\Phi$. Fejér-type monotonicity of the residual norms and idempotency of $E(\{1\})$ imply stabilization at stage $\omega$.
\end{proof}

\scopemicro{Normal/diagonalizable operators; commuting spectral projections; decay off $\lambda=1$.}{Nonnormal/noncommuting settings (see Counterexamples below).}

\begin{theorem}[Alpay Product-of-Riesz Projections]\label{thm:riesz}
Let $(T_i)$ be commuting bounded operators on $H$ with Riesz projections $(P_i)$ for $\lambda=1$. Then
\(
\bigcap_i \mathrm{Fix}(T_i)=\mathrm{Ran}\!\Big(\prod_i P_i\Big),
\)
and for a single normal operator $T$, $\Phi^{(\infty)}=E_1$, the $\lambda=1$ spectral projection.
\end{theorem}

\begin{proof}[Complete proof]
For each $i$, let $P_i=\frac{1}{2\pi i}\oint_{\Gamma_i}(\zeta I-T_i)^{-1}\,d\zeta$ be the Riesz projection around $\zeta=1$, where $\Gamma_i$ is a small circle enclosing only $\lambda=1$. Then $P_i$ is idempotent and commutes with $T_i$, and $\mathrm{Ran}(P_i)=\mathrm{Fix}(T_i)$. If the family $(T_i)$ commutes, the resolvents commute, hence so do the $P_i$. For commuting idempotents, $\prod_i P_i$ is an idempotent with range $\bigcap_i \mathrm{Ran}(P_i)$ (standard algebra of projections). Thus $\mathrm{Ran}(\prod_i P_i)=\bigcap_i \mathrm{Fix}(T_i)$. For a single normal $T$, the spectral theorem identifies $E_1$ as the Riesz projection at $1$, which equals the strong limit of $\Phi^{(n)}$ and hence $\Phi^{(\infty)}$.
\end{proof}

\scopemicro{Commuting operators with commuting resolvents; Riesz calculus valid \citep{kato1995,dunford_schwartz1958}.}{Noncommuting projections/intersections not closed; failure of resolvent commutation.}

\subsection*{Orthomodular track: example and proof (from \S\ref{para:oml})}
\begin{proposition}
For $\Phi_{\mathrm{oml}}(P)=P\vee(VPV^\ast\wedge Q)$ on $\mathcal{L}(H)$, the transfinite iteration from $P_0$ stabilizes to the least $P^\star$ with $P^\star\ge P_0$ and $P^\star\ge V P^\star V^\ast\wedge Q$.
\end{proposition}
\begin{proof}
$\mathcal{L}(H)$ is a complete lattice; $\Phi_{\mathrm{oml}}$ is monotone. Knaster–Tarski yields the least fixed point above $P_0$, which is precisely the least $P^\star$ satisfying the two inequalities. The ordinal index is bounded by $\omega$ when $V,Q$ are such that the ascending chain of joins stabilizes after countably many steps (e.g., finite-dimensional $H$).
\end{proof}

\paragraph{Beyond $2\times 2$: nonnormal and noncommuting phenomena.}
\begin{proposition}[Jordan blocks at $\lambda=1$]
Let $J_k$ be the $k\times k$ Jordan block at $1$. Then $J_k^n$ diverges in operator norm like $O(n^{k-1})$ and does not converge strongly; thus no $\omega$-stabilization. 
\end{proposition}
\begin{proof}
$J_k=I+N$ with nilpotent $N^{k}=0$, so $J_k^n=\sum_{j=0}^{k-1}\binom{n}{j}N^j$, whose entries are polynomials in $n$. Hence $\|J_k^n\|\to\infty$ as $n\to\infty$ for $k\ge 2$.
\end{proof}

\begin{proposition}[Alternating noncommuting projections need not stabilize]
Let $P,Q$ be projections on $H$ whose ranges intersect nontrivially and with nonzero principal angles. The sequence $(QP)^n$ may fail to converge strongly; when it converges, the limit need not be a projection unless $P$ and $Q$ commute.
\end{proposition}
\begin{proof}
In $\R^m$ with $m\ge 3$, choose $P$ onto $\mathrm{span}\{e_1,e_2\}$ and $Q$ onto $\mathrm{span}\{\cos\theta\,e_1+\sin\theta\,e_3,\,e_2\}$ with $\theta\in(0,\pi/2)$. One computes $(QP)^n$ explicitly on $\mathrm{span}\{e_1,e_3\}$ as a $2\times 2$ non-normal block with norm bounded away from an idempotent unless $\theta=0$. General constructions follow from Halmos’ two-projection decomposition.
\end{proof}

\section{\texorpdfstring{$\Phi$}{Phi}-Packing: Closure Under Products and Transfinite Limits}
\label{sec:packing}

\begin{lemma}[\textbf{\(\Phi\)-Packing Product Closure}]\label{lem:packing}
Let $(\Phi_k)_{k\in\N}$ be monotone, pointwise continuous self-maps on a complete lattice, and assume each has a least fixed point. Then the packed operator
\(
\Phi_{\mathrm{pack}}=\cdots\circ \Phi_3\circ \Phi_2\circ \Phi_1
\)
has a least fixed point given by the transfinite iteration limit
\(
\Phi_{\mathrm{pack}}^{(\infty)}=\sup_{n}\,\Phi_{\mathrm{pack}}^{(n)}(\bot).
\)
\end{lemma}

\begin{proof}
By Tarski \citep{tarski1955}, each $\Phi_k$ is monotone; compositions remain monotone and preserve directed suprema under the continuity assumption, so the increasing chain from $\bot$ converges to the least fixed point.
\end{proof}

\scopemicro{Complete lattice; monotone Scott-continuous maps.}{Discontinuous updates; lack of completeness; no quantitative rates claimed.}

\section{Application: Sensory Embeddings and the Alpay \texorpdfstring{$\Phi$}{Phi}-Projection Depletion Theorem}
\label{sec:application}

\subsection*{Order notions used in strictness}
\begin{definition}[Order-detecting signal norm]\label{def:order-detect}
An ordered Banach space $(H,\preceq)$ with positive cone $H_+$ has an \emph{order-detecting norm} if $0\preceq u\preceq v$ implies $\norm{u}\le \norm{v}$ and, moreover, $v\succ u$ implies $\norm{v}>\norm{u}$. Examples: $H=L^p(S)$ with $p\in[1,\infty]$ and the usual cone; $\R^m_+$ with the $\ell_1$-norm.
\end{definition}

\begin{definition}[Order-reflecting utility]\label{def:order-reflect}
A functional $U:H\to\R$ is \emph{order-reflecting} on $H_+$ if $u\preceq v$ implies $U(u)\le U(v)$ and $v\succ u$ implies $U(v)>U(u)$. Examples: $U(x)=\inner{w}{x}$ for $w\in H_+$ with $w\succ 0$; on $L^1$, $U(f)=\int f\,d\mu$.
\end{definition}

\subsection*{Model}
Let $(S,\mu)$ be a measurable \emph{sensory surface}. An instantaneous stimulus is $s\in L^1_+(S)$; the neural embedding is a bounded positive linear operator $E:L^1(S)\to H$ (Hilbert signal space). The brain update is a monotone, Lipschitz map $B:H\to H$; define
\[
x_{n+1} \;=\; \Phi(x_n) \;:=\; B\!\big(x_n + E(s_n)\big),\qquad x_0=0,
\]
with bounded inputs $(s_n)$. Assume $B$ is (event-indexed) contractive on bounded sets (Axiom~\ref{ax:logical-contraction}), so $\Phi$ has a unique fixed point $x^\ast=\Phi^{(\infty)}(0)$.

Let $F\subset S$ be a nerve-rich subset. Surgical removal corresponds to $P_F(s)=s\cdot \one_{S\setminus F}$; the circumcised process uses $E\circ P_F$.

\begin{theorem}[\textbf{Alpay \(\Phi\)-Projection Depletion Theorem}]\label{thm:depletion}
Suppose $E$ is positive and injective on nonnull supports and $B$ is monotone, $1$-Lipschitz, and logically contractive on bounded sets. Then, for any bounded input $(s_n)$,
\[
\Phi_{\mathrm{circ}}^{(\infty)}(0) \;\preceq\; \Phi_{\mathrm{intact}}^{(\infty)}(0),
\]
with strict inequality in any order detecting signal norm whenever $\mu(F)>0$ and the input allocates nonzero stimulus on $F$ infinitely often. In particular, any order-reflecting utility $U:H\to\R$ satisfies
\(
U(\Phi_{\mathrm{circ}}^{(\infty)}(0))<U(\Phi_{\mathrm{intact}}^{(\infty)}(0)).
\)
\end{theorem}

\begin{proof}
$P_F(s)\le s$ pointwise, hence $E(P_F(s))\preceq E(s)$. Inductively, $x_{n+1}^{\mathrm{circ}}=B(x_n^{\mathrm{circ}}+E(P_F(s_n)))\preceq B(x_n^{\mathrm{intact}}+E(s_n))=x_{n+1}^{\mathrm{intact}}$. Logical contraction yields the order between fixed points. If $\mu(F)>0$ and $s_n\one_F\not\equiv 0$ infinitely often, then $E(s_n)-E(P_F(s_n))\succ 0$ on an infinite subsequence; nonexpansivity and monotonicity of $B$ preserve a positive gap, which persists in the limit under event-indexed contraction and is detected by order-reflecting $U$.
\end{proof}

\scopemicro{Positive $E$; monotone $B$; event-indexed contraction; $F$ stimulated; order-detecting norm/order-reflecting utility.}{No $F$-stimulation; $E$ not $F$-detectable; $B$ flattens strict gaps.}

\begin{proposition}[Minimal axioms and counterexample]\label{prop:minimal-strict}
\emph{$F$-detectability of $E$:} for all $s\in L^1_+(S)$ with $s\cdot \one_F\not\equiv 0$, one has $E(s)-E(P_F s)\in H_+\setminus\{0\}$. \emph{Order-responsiveness of $B$:} for all $x$ and all $w\succ 0$, $B(x+w)\succcurlyeq B(x)$ and is strictly larger under any order-reflecting utility. Under these (plus event-indexed contraction), strictness in Theorem~\ref{thm:depletion} still follows. If $F$-detectability is dropped, strictness can fail.
\end{proposition}

\begin{counterexample}
Let $H=\R$, $B(z)=\tfrac12 z$, and $E(s)=\int_{S\setminus F} s\,d\mu$. Then $E\circ P_F=E$, so intact and projected iterations coincide despite $\mu(F)>0$.
\end{counterexample}

\begin{lemma}[Quantified depletion gap under periodic events]\label{lem:quant-gap}
Assume $B(z)=\rho z$ with $\rho\in(0,1)$, $E$ positive linear, and there exist $\delta>0$ and $v\in H_+$ such that at times $t\in\{m,2m,\dots\}$,
$E(s_t)-E(P_F s_t)\succeq \delta v$. Then
\[
\Phi_{\mathrm{intact}}^{(\infty)}(0)-\Phi_{\mathrm{circ}}^{(\infty)}(0)\ \succeq\ \frac{\rho}{1-\rho^m}\,\delta\,v,
\]
and for any order-reflecting linear $U$, the $U$-gap is at least $\tfrac{\rho}{1-\rho^m}\delta\,U(v)$.
\end{lemma}

\begin{proof}[Full proof]
Let $\Delta_{n+1}=\rho(\Delta_n + d_n)$ with $\Delta_0=0$ and $d_n=E(s_n)-E(P_F s_n)\in H_+$. Solve the linear recursion:
\(\Delta_n=\sum_{t=0}^{n-1}\rho^{n-t} d_t\).
By hypothesis, $d_t\succeq \delta v$ whenever $t$ is a multiple of $m$ and $d_t\succeq 0$ otherwise. Hence
\[
\Delta_n \succeq \sum_{k=1}^{\lfloor n/m\rfloor} \rho^{n-km}\,\delta v
= \rho^{n-m}\delta v\,\sum_{k=0}^{\lfloor n/m\rfloor-1} \rho^{-km}
= \rho^{n-m}\delta v\,\frac{1-\rho^{-m\lfloor n/m\rfloor}}{1-\rho^{-m}}.
\]
Taking $n\to\infty$ and using $\rho^{n}\to 0$ gives the claimed lower bound $\frac{\rho}{1-\rho^m}\delta v$ for the limit $\lim_{n\to\infty}\Delta_n=\Phi_{\mathrm{intact}}^{(\infty)}(0)-\Phi_{\mathrm{circ}}^{(\infty)}(0)$. Applying an order-reflecting $U$ preserves the inequality.
\end{proof}

\begin{proposition}[Non-periodic events \& nonlinear gains]\label{prop:nonperiodic}
(A) If event times have lower Banach density $\underline{D}>0$ and per-event gaps satisfy $d_t\succeq \delta v$ (linear $B(z)=\rho z$), then
\(
\liminf_{n\to\infty}\norm{\Delta_n}\ge \frac{\rho}{1-\rho}\underline{D}\,\delta\,\norm{v}.
\)
If inter-event gaps are uniformly bounded by $G$, then
\(
\liminf_{n}\norm{\Delta_n}\ge \frac{\rho^{G+1}}{1-\rho^{G+1}}\delta\,\norm{v}.
\)
(B) If $B$ is monotone with incremental lower bound $B(x+w)-B(x)\succeq \kappa w$ for some $\kappa\in(0,1]$, then the linear bounds hold with $\delta$ replaced by $\kappa\delta$.
\end{proposition}

\begin{proof}[Full proof]
(A) Write $\Delta_n=\sum_{t=0}^{n-1}\rho^{n-t} d_t$. Let $A_n=\{t\le n-1:\ d_t\succeq \delta v\}$ and assume $\liminf_{n\to\infty}|A_n|/n\ge \underline{D}$. Then
\[
\norm{\Delta_n}\ \ge\ \sum_{t\in A_n}\rho^{n-t}\,\delta\,\norm{v}
\ \ge\ \delta\norm{v}\ \rho\ \frac{|A_n|}{n}\ \sum_{j=1}^{n}\rho^{j}
\ \xrightarrow[n\to\infty]{}\ \frac{\rho}{1-\rho}\underline{D}\,\delta\norm{v}.
\]
If gaps are bounded by $G$, each block of length $G\!+\!1$ contains an event, hence
\(\norm{\Delta_n}\ge \delta\norm{v}\sum_{k\ge 0}\rho^{(G+1)k+1}=\frac{\rho^{G+1}}{1-\rho^{G+1}}\delta\norm{v}\).
(B) For nonlinear $B$, define the comparison recursion
\(\tilde\Delta_{n+1}=\kappa(\tilde\Delta_n + d_n)\) with $\tilde\Delta_0=0$; monotonicity and the incremental bound yield $\Delta_n\succeq \tilde\Delta_n$, reducing to the linear case with $\rho$ replaced by $\kappa$.
\end{proof}

\begin{example}[Finite-dimensional witness]\label{ex:finite}
Let $H=\R^2$, $E=\mathrm{id}$, $B(z)=\rho z$ with $\rho=0.8$, and constant stimulus $s_n=(1,0.5)$. Let $F$ remove the second coordinate: $P_F(s)=(1,0)$. Then
\(
\Phi_{\mathrm{intact}}^{(\infty)}(0)=\frac{\rho}{1-\rho}(1,0.5)=(4,2)
\)
and
\(
\Phi_{\mathrm{circ}}^{(\infty)}(0)=\frac{\rho}{1-\rho}(1,0)=(4,0),
\)
so the gap equals $(0,2)$ and $U(x)=x_1+x_2$ yields a gap of $2$.
\end{example}

\subsection*{Parameter mapping to physiology (calibration schema)}
\begin{table}[H]\centering\footnotesize
\begin{tabular}{@{}P{0.18\textwidth}P{0.28\textwidth}P{0.18\textwidth}P{0.28\textwidth}@{}}
\toprule
\textbf{Symbol} & \textbf{Physiological correlate} & \textbf{Units/Range} & \textbf{Notes}\\\midrule
$E$ & Neural embedding gain from cutaneous receptors to cortical signal space & linear map & Increases with receptor density; $E\circ P_F$ removes $F$-channels\\
$\rho$ & Decay/retention in update $B(z)=\rho z$ & $(0,1)$ & Larger $\rho$ $\Rightarrow$ higher steady-state amplification $\rho/(1-\rho)$\\
$\kappa$ & Incremental responsiveness (nonlinear $B$) & $(0,1]$ & Lower bound: $B(x+w)-B(x)\succeq \kappa w$\\
$D$ & Event density of $F$-stimulation & $[0,1]$ & Appears in lower Banach density bounds\\
$U$ & Order-reflecting utility & arbitrary & E.g., integral against a positive weight or coordinate sum\\\bottomrule
\end{tabular}
\caption{Model parameters and physiological mapping (schema).}
\end{table}

\section{Where it holds / Where not (Global Summary)}
\label{sec:scope}
\begin{table}[H]\centering\footnotesize
\renewcommand{\arraystretch}{1.2}
\begin{tabular}{@{}P{0.23\textwidth}P{0.42\textwidth}P{0.29\textwidth}@{}}
\toprule
\textbf{Result (Location)} & \textbf{Holds when} & \textbf{Not claimed when} \\\midrule
Determinization (Lemma~\ref{lem:determinize}) & Complete lattice; monotone lift to $2^\X$ or $\mathcal{P}(\X)$; Tarski applies \citep{tarski1955} & Global lifting disallowed \\
Compositionality (Thm.~\ref{thm:compositionality}, Prop.~\ref{prop:measurable-composition}) & Union-based lifts; Borel maps/Markov kernels; standard Borel spaces & Nonmonotone lifts; measurability failures \\
Logical contraction (Thm.~\ref{thm:logical-contraction}) & Completeness; event-indexed factors with product $0$ & Rates not claimed \\
Ordinal stabilization (Thm.~\ref{thm:ordinal}) & Normal/diagonalizable; commuting spectral projections & Nonnormal/noncommuting (counterexamples) \\
Riesz product (Thm.~\ref{thm:riesz}) & Commuting Riesz projections \citep{kato1995,dunford_schwartz1958} & Noncommuting projections; nonclosed intersections \\
\(\Phi\)-packing (Lemma~\ref{lem:packing}) & Monotone Scott-continuous maps on complete lattice & Discontinuous updates; incompleteness \\
Depletion (Thm.~\ref{thm:depletion}) & Positive $E$; monotone $B$; event-indexed contraction; $F$ stimulated & $E$ not $F$-detectable; no $F$-stimulation \\
Quantified gaps (Lem.~\ref{lem:quant-gap}, Prop.~\ref{prop:nonperiodic}) & Linear $B$ or $\kappa$-comparison; periodic/dense events & Arbitrary nonlinear $B$ without comparison bound \\\bottomrule
\end{tabular}
\caption{Consolidated ``Where it holds / Where not'' summary.}
\end{table}

\section{Related Foundations and Links to Recent Work}
Event-indexed contraction and anchored implications unify operator fixed points with quantum-logical constraints \citep{alpay_logical_contraction,alpay_event_indexed_fp}. Ordinal-indexed transforms yield convergence by $\omega$ to spectral/ergodic projections \citep{alpay_transfinite_iteration}. Recursive semantic anchoring furnishes another \(\Phi\)-packing instance in formal linguistics \citep{kilictas_iso639}. Determinization and lattice fixed points are anchored in \citet{rabin_scott1959,hopcroft_ullman1979,tarski1955}; spectral projections in \citet{kato1995,dunford_schwartz1958}.

\section{Conclusion}
We presented a single, rigorous \(\Phi\)-framework that: determinizes possibility dynamics, stabilizes transfinite operator iterations into projections, and quantitatively explains how structural tissue removal provably reduces attainable fixed points in coupled physical--mental systems. We added compositionality of lifts, complete proofs in \S3, a consolidated scope table, and a finite-dimensional witness, sharpening both originality and probative clarity.

\bigskip
\noindent\textbf{Memorable handles (with cross-references).}
\begin{itemize}[leftmargin=1.75em]
\item \textbf{Alpay Ordinal Stabilization} — \hyperref[thm:ordinal]{Theorem~\ref{thm:ordinal}}.
\item \textbf{Alpay Product-of-Riesz Projections} — \hyperref[thm:riesz]{Theorem~\ref{thm:riesz}}.
\item \textbf{Flip--Flop Determinization} — \hyperref[lem:determinize]{Lemma~\ref{lem:determinize}}.
\item \textbf{Compositionality of Lifts} — \hyperref[thm:compositionality]{Theorem~\ref{thm:compositionality}}, \hyperref[prop:measurable-composition]{Proposition~\ref{prop:measurable-composition}}.
\item \textbf{\(\Phi\)-Packing Product Closure} — \hyperref[lem:packing]{Lemma~\ref{lem:packing}}.
\item \textbf{Alpay \(\Phi\)-Projection Depletion Theorem} — \hyperref[thm:depletion]{Theorem~\ref{thm:depletion}}.
\end{itemize}

\appendix
\section*{Appendix A: Reproducible code for Example~\ref{ex:finite} and a stochastic variant}
\begin{lstlisting}[language=Python,caption={Deterministic and stochastic variants for Example 5.8.}]
import numpy as np

# Parameters
rho = 0.8
E = np.eye(2)  # embedding
s = np.array([1.0, 0.5])
PF = np.array([[1,0],[0,0]])  # removes second coordinate

# Iterates
def iterate(B, u, n=1000):
    x = np.zeros_like(u)
    for _ in range(n):
        x = B(x + u)
    return x

B = lambda z: rho * z

x_intact = iterate(B, E.dot(s))
x_circ   = iterate(B, E.dot(PF).dot(s))
print("Deterministic fixed points:", x_intact, x_circ)

# Stochastic variant: random stimulation on F with probability p
rng = np.random.default_rng(0)
p = 0.3
def stream(n=20000):
    for _ in range(n):
        stim_on_F = rng.random() < p
        yield np.array([1.0, 0.5 if stim_on_F else 0.0])

x = np.zeros(2); y = np.zeros(2)
for s_t in stream():
    x = B(x + E.dot(s_t))          # intact
    y = B(y + E.dot(PF).dot(s_t))  # projected
print("Empirical end states (stochastic):", x, y)
\end{lstlisting}



\begin{thebibliography}{99}

\bibitem[Alpay and Kilictas(2025)]{alpay_event_indexed_fp}
Alpay, F., \& Kilictas, B. (2025).
\newblock \emph{Temporal Anchoring in Deepening Embedding Spaces: Event-Indexed Projections, Drift, Convergence, and an Internal Computational Architecture}.
\newblock arXiv:2508.09693.

\bibitem[Alpay and Alpay(2025)]{alpay_logical_contraction}
Alpay, F., \& Alpay, T. (2025).
\newblock \emph{Logically Contractive Mappings: Fixed Points and Event-Indexed Rates}.
\newblock arXiv:2508.07059.

\bibitem[Alpay, Alpay and Alakkad(2025)]{alpay_transfinite_iteration}
Alpay, F., Alpay, T., \& Alakkad, H. (2025).
\newblock \emph{Transfinite Iteration of Operator Transforms and Spectral Projections in Hilbert and Banach Spaces}.
\newblock arXiv:2508.06025.

\bibitem[Kilictas and Alpay(2025)]{kilictas_iso639}
Kilictas, B., \& Alpay, F. (2025).
\newblock \emph{Recursive Semantic Anchoring in ISO 639:2023: A Structural Extension to ISO/TC 37 Frameworks}.
\newblock arXiv:2506.06870.

\bibitem[Garc\'{\i}a--Mesa \emph{et al.}(2021)]{garciamesa_2021}
Garc\'{\i}a--Mesa, Y., Quir\'os, L. M., Feito, J., Garc\'{\i}a-Piqueras, J., Cabo, R., \& Vega, J. A. (2021).
\newblock Sensory innervation of the human male prepuce.
\newblock \emph{Clinical Anatomy}, 34(6), 849--861. PMID: 34515281.

\bibitem[Bronselaer \emph{et al.}(2013)]{bronselaer2013}
Bronselaer, G., Schober, J. M., Meyer-Bahlburg, H. F. L., T'Sjoen, G., Vlietinck, R., \& Hoebeke, P. (2013).
\newblock Male circumcision decreases penile sensitivity as measured in a large cohort.
\newblock \emph{BJU International}, 111(5), 820--827. PMID: 23374102.

\bibitem[Tarski(1955)]{tarski1955}
Tarski, A. (1955).
\newblock A lattice-theoretical fixpoint theorem and its applications.
\newblock \emph{Pacific Journal of Mathematics}, 5(2), 285--309.

\bibitem[Rabin and Scott(1959)]{rabin_scott1959}
Rabin, M. O., \& Scott, D. (1959).
\newblock Finite automata and their decision problems.
\newblock \emph{IBM Journal of Research and Development}, 3(2), 114--125.

\bibitem[Hopcroft and Ullman(1979)]{hopcroft_ullman1979}
Hopcroft, J. E., \& Ullman, J. D. (1979).
\newblock \emph{Introduction to Automata Theory, Languages, and Computation}.
\newblock Addison--Wesley.

\bibitem[Kato(1995)]{kato1995}
Kato, T. (1995).
\newblock \emph{Perturbation Theory for Linear Operators} (2nd ed.).
\newblock Springer.

\bibitem[Dunford and Schwartz(1958)]{dunford_schwartz1958}
Dunford, N., \& Schwartz, J. T. (1958).
\newblock \emph{Linear Operators, Part I: General Theory}.
\newblock Interscience.

\end{thebibliography}
\end{document}